\documentclass[a4paper,reqno,11pt]{amsart}
\usepackage{
amssymb,
amsmath,
amsthm,
eucal,
dsfont,
multicol,
mathrsfs}

\makeatletter
 
  \@addtoreset{equation}{section}
 \makeatother
\theoremstyle{plain}

\newtheorem{theorem}{Theorem}[section]
\newtheorem{lemma}[theorem]{Lemma}

\newtheorem{corollary}[theorem]{Corollary}
\theoremstyle{definition}

\newtheorem{remark}[theorem]{Remark}
\newtheorem{example}[theorem]{Example}

\newcommand{\norm}[1]{{||#1||}}

\newcommand{\wtilde}[1]{{\widetilde{#1}}}

\def\Id{\mathop{\mathrm{Id}}\nolimits}

\def\Ran{\mathop{\mathrm{Ran}}\nolimits}

\def\Im{\mathop{\mathrm{Im}}\nolimits}

\def\sgn{\mathop{\mathrm{sgn}}\nolimits}
\def\R{{\mathbb{R}}}
\def\Z{{\mathbb{Z}}}

\def\C{{\mathbb{C}}}

\def\F{{\mathcal{F}}}
\def\H{{\mathcal{H}}}

\def\<{{\langle}}
\def\>{{\rangle}}

\def\ep{{\varepsilon}}
\def\ds{\displaystyle}
\DeclareMathOperator*{\slim}{s-lim}

\title[Wave operators on Sobolev spaces]{Wave operators on Sobolev spaces}

\author{Haruya Mizutani}
\address{Department of Mathematics, Graduate School of Science, Osaka University, Toyonaka, Osaka 560-0043, Japan.}
\email{haruya@math.sci.osaka-u.ac.jp}

\begin{document}
\date{\today}

\begin{abstract}
We provide a simple sufficient condition in an abstract framework to deduce the existence and completeness of wave operators (resp. modified wave operators) on Sobolev spaces from the existence and completeness of the usual wave operators (resp. modified wave operators). We then give some examples of Schr\"odinger operators for which our abstract result applies. An application to scattering theory for the nonlinear Schr\"odinger equation with a potential is also given. 
\end{abstract}

\maketitle

\section{Introduction}
Let $\H$ be a Hilbert space with norm $\norm{\cdot}$ and inner product $\<\cdot,\cdot\>$, and $H_0,H$ two self-adjoint operators on $\H$. The operator norm on $\H$ is also denoted by $\norm{\cdot}$. Define the Sobolev space $\H^s:=\<H_0\>^{-s/2}\H$ of order $s$ with norm 
$\norm{f}_s:=\norm{\<H_0\>^{s/2}f},
$
where $\<\cdot\>=(1+|\cdot|^2)^{1/2}$. A typical example is the Schr\"odinger operators $H_0=-\Delta$ and $H=H_0+V$ on $L^2(\R^n)$ with a real-valued potential $V:\R^n\to\R$
in which case $\H^s=H^s(\R^n)$ is the standard $L^2$-based Sobolev space. 

We regard $H_0$ as a free operator and study the scattering theory for the pair $(H_0,H)$, namely the asymptotic behavior of $e^{-itH}P_{\mathrm{ac}}(H)$ in the limit $t\to\infty$ as a perturbation of the free motion $e^{-itH_0}P_{\mathrm{ac}}(H_0)$ (or a suitable modified free motion in the long-range case), where $P_{\mathrm{ac}}(H)$ denotes the projection onto $\H_{\mathrm{ac}}(H)$, the absolutely continuous spectral subspace of $H$. In particular, we are interested in the asymptotic behavior of $e^{-itH}P_{\mathrm{ac}}(H)$ in $\H^s$ for $s\neq0$. To this end, we consider the existence of the {\it wave and inverse wave operators}
$$
W_{\pm,s}=\slim_{t\to \pm\infty}e^{itH}e^{-itH_0}P_{\mathrm{ac}}(H_0),\ \Omega_{\pm,s}=\slim_{t\to \pm\infty}e^{itH_0}e^{-itH}P_{\mathrm{ac}}(H)\quad\text{in}\quad \H^s.
$$
Under the norm equivalence condition $\H^s\cong D(\<H\>^{s/2})$ (see the condition (H1) below), the existence of $W_{\pm,s}$ implies, for any $u_0\in \H^s$ there exist $u_\pm\in \H^s$ such that
$$
\norm{e^{-itH}u_\pm -e^{-itH_0}P_{\mathrm{ac}}(H_0)u_0}_s\to 0,\quad t\to \pm\infty,
$$
while the existence of $\Omega_{\pm,s}$ implies, for any $u_0\in \H^s$ there exist $u_\pm\in \H^s$ such that
$$
\norm{e^{-itH}P_{\mathrm{ac}}(H)u_0-e^{-itH_0}u_\pm}_s\to 0,\quad t\to \pm\infty.
$$
The existence of the usual wave operators $W_{\pm,0},\Omega_{\pm,0}$ has developed from almost the beginning of mathematical analysis of Quantum Mechanics and there is a huge literature (see, for instance, monographs \cite{ReSi,DeGe,Yaf}). On the other hand, to the best knowledge of the author, the case $s\neq0$ has attracted less interest. However, the wave operators  on Sobolev spaces appear naturally in the study of scattering theory for the nonlinear Schr\"odinger equation
$$(i\partial_t +\Delta-V)u=N(u);\quad u|_{t=0}=u_0\in H^s(\R^n).$$
Indeed, it often happens that, with a suitable nonlinear term $N(u)$, the nonlinear Schr\"odinger equation has a global solution for $u_0\in H^s(\R^n)$ with some $s>0$, but not for $u_0\in L^2(\R^n)$. It is then natural to consider the scattering theory in the same topology of the initial data space. 

This short note provides a simple sufficient condition in an abstract framework to deduce the existence and completeness of $W_{\pm,s}$ from the existence and completeness of $W_{\pm,0}$. Some applications to Schr\"odinger operators with potentials are also given. We also give an application to the scattering theory for a nonlinear Schr\"odinger equation with a linear potential. 

The paper is organized as follows. Section \ref{section_main} is devoted to the main theorem and its proof. In Section \ref{section_example}, we give some applications of the main result to Schr\"odinger operators and the nonlinear scattering theory. Appendix \ref{appendix_A} consists of some supplementary lemmas.

\section{Main result}
\label{section_main}
In what follows, we use the following notation. $\mathbb B(X,Y)$ (resp. $\mathbb B_\infty (X,Y)$) denotes the Banach space of bounded (resp. compact) operators from $X$ to $Y$. Let $\mathbb B(X)=\mathbb B(X,X)$ and $\mathbb B_\infty(X)=\mathbb B_\infty(X,X)$. For positive constants $A,B$, $A\lesssim B$ means $A\le CB$ with some constant $C>0$. 

Let us fix $s\in \R$ and consider the following series of assumptions. 

\begin{itemize}
\item[(H1)] $\H^s\cong D(\<H\>^{s/2})$. In other words, for all $f\in \H^s$ and $g\in D(\<H\>^{s/2})$, $$\norm{\<H\>^{s/2}f}\lesssim \norm{\<H_0\>^{s/2}f},\quad \norm{\<H_0\>^{s/2}g}\lesssim \norm{\<H\>^{s/2}g}. $$
\item[(H2)] $\ds (H_0-z)^{-1}-(H-z)^{-1}\in \mathbb B_\infty(\H^s)\cup \mathbb B_\infty(\H,\H^s)$ for all $z\in \C\setminus\R$. 
\item[(H3)] There exists a family of unitary operators $\{U_0(t)\}_{t\in \R}$ on $\H$ which commutes with $H_0$, that is $U_0(t)\H^2\subset \H^2$ and $[U_0(t),H_0]=0$ on $\H^2$  for all $t\in \R$.
\item[(H4)] For any $f\in \H_{\mathrm{ac}}(H_0)$, $U_0(t)f\to0$ weakly in $\H$ as $t\to \pm\infty$. Moreover, the wave operators
$
\ds W_\pm:=\slim_{t\to \pm\infty}e^{itH}U_0(t)P_{\mathrm{ac}}(H_0)$ in $\H$ exist. 
\item[(H5)] The inverse wave operators
$\ds \Omega_\pm:=\slim_{t\to \pm\infty}U_0(t)^*e^{-itH}P_{\mathrm{ac}}(H)$ in $\H$ exist. 
\end{itemize}

\begin{remark}
\label{remark_2_1}
(1) (H1) and (H2) imply  $[\psi(H_0)-\psi(H)]\varphi(H_0)\in \mathbb B_\infty(\H,\H^s)$ for any $\varphi,\psi\in C_0^\infty(\R)$ (see Lemma \ref{lemma_A_1} in Appendix \ref{appendix_A} below). \\
(2) By a standard approximation argument, we see that $U_0(t)$ also commutes with $\varphi(H_0)$ for any $\varphi\in L^\infty(\R)\cup C(\R)$. In particular, $U_0(t)$ is bounded on $\H^s$ uniformly in $t\in \R$. \\
(3) Under (H4), (H5) is equivalent to $\Ran W_\pm=\H_{\mathrm{ac}}(H)$ in which case $\Omega_\pm=(W_\pm)^*$ (see \cite{ReSi}). \\
(4) By the Riemann-Lebesgue lemma, for all $f\in \H_{\mathrm{ac}}(H)$, $e^{-itH}f\to0$ weakly in $\H$ as $t\to \pm\infty$. 
 \end{remark}

\begin{theorem}	
\label{theorem_1}
Let $s\in \R$. Under {\rm (H1)},  {\rm (H2)} and {\rm (H3)}, the following statements hold: \\
{\rm (1)} If {\rm (H4)} is satisfied then 
$\ds W_{\pm,s}:=
\slim_{t\to \pm\infty}e^{itH}U_0(t)P_{\mathrm{ac}}(H_0)$ 
in $\H^s$ exist. \\
{\rm (2)} If {\rm (H5)} is satisfied then
$\ds \Omega_{\pm,s}:=\slim_{t\to \pm\infty}U_0(t)^*e^{-itH}P_{\mathrm{ac}}(H)$ in $\H^s$ exist. 
 \end{theorem}

\begin{proof} We prove the existence of $W_{+,s}$ only, the proof of other statements being analogous. Let $u_0\in \H^s$ and $W(t)=e^{itH}U_0(t)P_{\mathrm{ac}}(H_0)$ which is bounded on $\H^s$ uniformly in $t$ by (H1) and (H3). We shall show that, for any sequence $t_n\to \infty$, $\{W(t_n)u_0\}$ is a Cauchy sequence in $\H^s$. Let us fix $\ep>0$ and $\varphi\in C_0^\infty(\R)$ be such that $\varphi\equiv1$ near origin and set $\varphi_R(\lambda)=\varphi(\lambda/R)$. Then
\begin{align*}
\norm{u_0-\varphi_R(H_0)u_0}_{s}^2=\int_{\sigma(H_0)} |1-\varphi(\lambda/R)|^2d\norm{E_{H_0}(\lambda)u_0}_s^2\to 0,\quad R\to \infty,
\end{align*}
by the dominated convergence theorem. In particular, there exists $R\ge1$ such that 
$$
\sup_{n,m}\norm{(W(t_n)-W(t_m))(u_0-\varphi_R(H_0)u_0)}_{s}\lesssim \norm{u_0-\varphi_R(H_0)u_0}_{s}< \ep.
$$
We thus may replace $u_0$ by $u_R:=\varphi_R(H_0)u_0$ without loss of generality. Choose $\psi \in C_0^\infty(\R)$ so that $\psi\varphi\equiv\varphi$.  Then
\begin{align*}
[W(t_n)-W(t_m)]u_R=\psi_R(H)[W(t_n)-W(t_m)]u_R+(1-\psi_R)(H)[W(t_n)-W(t_m)]u_R.
\end{align*}
Since $\norm{\<H_0\>^{s/2}\psi_R(H)}\lesssim R^{s/2}$ by (H1),  the second assumption in (H4) implies
\begin{align*}
\norm{\psi_R(H)[W(t_n)-W(t_m)]u_R}_{s}\lesssim R^{s/2}\norm{[W(t_n)-W(t_m)]u_R}\to 0,\quad m,n\to \infty.
\end{align*}
Moreover, since
$
(1-\psi_R)(H)\varphi_R(H_0)=[\psi_R(H_0)-\psi_R(H)]\varphi_R(H_0)
$, we have
\begin{align*}
&\norm{(1-\psi_R)(H)[W(t_n)-W(t_m)]u_R}_{s}\\
&\le \norm{[\psi_R(H_0)-\psi_R(H)]U_0(t_n)P_{\mathrm{ac}}(H_0)u_R}_{^s}+\norm{[\psi_R(H_0)-\psi_R(H)]U_0(t_m)P_{\mathrm{ac}}(H_0)u_R}_{s}
\to 0
\end{align*}
as $m,n\to \infty$, where we have used (H3) in the first inequality, the first condition in (H4) and Remark \ref{remark_2_1} (1) in the last step, respectively. Hence $\{W(t_n)u_R\}$ is a Cauchy sequence in $\H^s$. 
\end{proof}


\section{Application to Schr\"odinger equations}
\label{section_example}

Here we apply the above theorem to the scattering theory for Schr\"odinger equations on $\R^n$. Throughout this section, we set $\H=L^2(\R^n)$ and $H_0=-\Delta$ with $D(H_0)=H^2(\R^n)$ in which case $\H^s=H^s(\R^n)$. We first give some typical examples of potentials satisfying (H1) and (H2). 
\begin{example}
\label{example_1_4}
Let $\max(n/2,1)\le p<\infty$, and $p>1$ if $n=2$. Suppose that $V:\R^n\to \R$ belongs to $L^p(\R^n)+L^\infty_0(\R^n)$, where $L^\infty_0(\R^n)$ is the $L^\infty$-norm closure of $C_0^\infty(\R^n)$. Then $V$ is $H_0$-form compact (see Lemmas \ref{lemma_B_1} in Appendix \ref{appendix_A} below), having relative bound zero. By the KLMN theorem, we can define a self-adjoint operator  $H=H_0+V$ as the form sum such that $D(\<H\>^{1/2})\cong \H^1$. Then (H1) and (H2) hold for $s\in [-1,1]$. Indeed, the complex interpolation and a duality argument show $D(\<H\>^{s/2})\cong \H^s$ for $s\in [-1,1]$. For the part (H2), we compute
\begin{align*}
&\<H_0\>^{s/2}[(H_0-z)^{-1}-(H-z)^{-1}]\\
&=\<H_0\>^{s/2}(H_0-z)^{-1}V(H-z)^{-1}\\
&=\<H_0\>^{s/2+1/2}(H_0-z)^{-1}\cdot\<H_0\>^{-1/2}|V|^{1/2}\cdot \sgn V
|V|^{1/2}\<H_0\>^{-1/2}\cdot \<H_0\>^{1/2}(H-z)^{-1},
\end{align*}
where $\<H_0\>^{-1/2}|V|^{1/2}$ is compact and the other terms are bounded on $\H$ as long as $s\le1$. 
\end{example}

\begin{example}
Let $n\ge3$, $a>-(n-2)^2/4$ and $V(x)=a|x|^{-2}$. By Hardy's inequality $$\frac{(n-2)^2}{4}\int |x|^{-2}|f|^2\le \int |\nabla f|^2,\quad f\in C_0^\infty(\R^n),$$
$H=H_0+V$ defined as the Friedrichs extension of the quadratic form $\<(H_0+V)u,u\>$ on $C_0^\infty(\R^n)$ satisfies $D(\<H\>^{1/2})\cong\H^1$. Hence (H1) holds for $s\in [-1,1]$. Moreover, writing
\begin{align*}
&\<H_0\>^{s/2}(H-z)^{-1}|x|^{-2}(H_0-z)^{-1}\\
&=\<H_0\>^{s/2}\<H\>^{-1/2}\cdot\<H\>^{1/2}(H-z)^{-1}\<H\>^{1/2}\,\times\\
&\quad\times \<H\>^{-1/2}\<H_0\>^{1/2}\cdot \<H_0\>^{-1/2}|x|^{-1}\cdot|x|^{-1}\<H_0\>^{-1}\cdot \<H_0\>(H_0-z)^{-1}
\end{align*}
we see that $|x|^{-1}\<H_0\>^{-1}\in\mathbb B_\infty(\H)$ since $|x|^{-1}\in L^{p}(\R^n)+L_0^\infty(\R^n)$ with some $n/2<p<n$, otherwise are bounded on $\H$ if $s\le1$. Therefore $(H_0-z)^{-1}-(H-z)^{-1}\in\mathbb B_\infty(\H,\H^s)$ for $s\le1$.
\end{example}

We next provide some examples to which our abstract theorem applies. 

\begin{example}[Short range potential]
\label{example_3_3}
Assume that $V:\R^n\to \R$ satisfies one of the following: 
\begin{itemize}
\item $n=1$ and $V\in L^1(\R)$;
\item $n\ge2$ and $V=V_0+V_1+\cdots+ V_N$, where 
\begin{itemize}
\item[$\bullet$]$|V_0(x)|\lesssim \<x\>^{-\rho}$ for  some $\rho>1$ and, \item[$\bullet$] $V_j\in L^{p_j}(\R^n)$ for $p_j$ satisfying $n/2\le p_j\le (n+1)/2$ and $p_j>1$ for $j=1,...,N$;
\end{itemize}
\item  $n\ge3$ and  $V(x)=a|x|^{-2}$ with $a>-(n-2)^2/4$. 
\end{itemize}
Then (H1) and (H2) hold for $s\in [-1,1]$ by the above examples. Moreover, the usual wave operators $
\ds W_\pm=\slim_{t\to\pm\infty} e^{itH}e^{-itH_0}$ in $\H$ exist and are complete. We refer to Reed-Simon \cite[Theorem XI.30]{ReSi} for the first case, Ionescu-Schlag \cite{IoSc} for the second case, respectively. For the last case, the existence and completeness of $W_\pm$ follow from the fact that $|x|^{-1}$ is both $H_0$-smooth and $H$-smooth in the sense that $|x|^{-1}e^{-itH_0},|x|^{-1}e^{-itH}\in \mathbb B(L^2(\R^n),L^2(\R^{1+n}))$ (see \cite{BPST2}) and the smooth perturbation theory by Kato \cite{Kat}. Hence, for all $s\in [-1,1]$, Theorem \ref{theorem_1} with  $U_0(t)=e^{-itH_0}$ applies. 
\end{example}

\begin{example}[Long-range potential]
\label{example_3_4}
Let $n\ge1$ and $V=V^S+V^L:\R^n\to \R$ be such that 
\begin{itemize}
\item $V^S$ is $H_0$-compact and $\<H_0\>^{-1}V(x)\mathds1_{[1,\infty)}(|x|/r)\<H_0\>^{-1/2}\in L^2([0,\infty)_r;L^2(\R^n_x))$;
\item $V^L\in C^\infty(\R^n)$  and, with some $\rho>0$,  
$|\partial_x^\alpha V(x)|\le C_\alpha \<x\>^{-\rho-|\alpha|}$ on $\R^n$ for all $\alpha\in \Z_+^n$.
\end{itemize}
Then (H1) and (H2) hold for $s\in [-2,2]$ since $V$ is $H_0$-compact. Moreover, there exists $S\in C^\infty(\R\times \R^n;\R)$ such that
\begin{itemize}
\item $\partial_\xi^\alpha(S(t,\xi)-t|\xi|^2)=o(t)$ as $|t|\to \infty$ for $|\alpha|\le2$;
\item for any $\ep>0$, there exists $T_\ep>0$ such that for $|\xi|\ge \ep$, $|t|\ge T_\ep$, $S(t,\xi)$ solves
$$
\partial_t S(t,\xi)=|\xi|^2+V^L(\nabla_\xi S(t,\xi));
$$
\item the modified wave operators
$
\ds W_\pm^S=\slim_{t\to\pm\infty} e^{itH}e^{-iS(t,D)}$ in $\H$ 
exist and are complete. 
\end{itemize}
We refer to \cite[Theorem 4.7.1]{DeGe}. Moreover, it follows from the above asymptotics of $S$ that for any $f\in \H$, $e^{-iS(t,D)}f\to 0$ weakly in $\H$ as $t\to \infty$ by the stationary phase theorem. Theorem \ref{theorem_1} with $U_0(t)=e^{-iS(t,D)}$ thus applies for $s\in [-2,2]$. 
\end{example}

\begin{remark}
As a typical example, $V(x)=-Z|x|^{-\mu}$ with $Z\in \R$ and $\mu\in (0,1]$ satisfies the above condition in Example \ref{example_3_4} if $n\ge3$.
\end{remark}

\begin{example}[Point interaction]
\label{example_3_6}
Let $\alpha\in \R$ and $H_{\alpha}=-\partial_x^2+\alpha \delta$ be the Schr\"odinger operator with a delta potential in $\R$. More precisely, $H_\alpha$ is defined as follows: 
$$
H_\alpha f=-f'',\quad D(H_\alpha)=\{f\in H^1(\R)\cap H^2(\R\setminus\{0\})\ |\ f'(0+)-f'(0-)=\alpha f(0)\}. 
$$
Note that $H_\alpha|_{\alpha=0}$ coincides with $H_0f=-f''$ with $D(H_0)=\H^2$. Then the form domain of $H_\alpha$ is $\H^1$ and $K_z:=(H_\alpha-z)^{-1}-(H_0-z)^{-1}$ is a rank one operator with the kernel 
$$
K_z(x,x')=-2\alpha\sqrt z(i\alpha+2\sqrt z)^{-1}e^{i\sqrt z(|x|+|x'|)}
$$
where we take a branch of $\sqrt z$ so that $\Im \sqrt z>0$ (see \cite[Chapter 1.3]{AGHH}). In particular, $K_z$ is in the trace class. We decompose $\<H_0\>^{s/2}K_z\<H_0\>^{-s/2}=K_z^2K_z^1$, where $K_z^1:=e^{i\sqrt z|x|/2}\<H_0\>^{-s/2}$ and 
\begin{align*}
K_z^2f(x):=-2\alpha\sqrt z(i\alpha+2\sqrt z)^{-1}\<H_0\>^{s/2}\int e^{i\sqrt z(|x|+|x'|/2)}f(x')dx. 
\end{align*}
Then $K^1_z$  is compact on $\H$ if $s>0$ and $K_z^2$ is bounded on $\H$ if $s<3/2$ since $\F(e^{-|x|})(\xi)=O(\<\xi\>^{-2})$. Hence $K_z\in \mathbb B_\infty(\H^s)$ if $0<s<3/2$. By the duality, $K_z\in \mathbb B_\infty(\H^s)$ for $-3/2<s<0$. Finally, by the scattering theory for trace class operators (see \cite{Yaf}), the usual wave operators on $\H$ exist and are complete. Hence Theorem \ref{theorem_1} with $U_0(t)=e^{-itH_0}$ applies for $s\in [-1,1]$. \end{example}

We conclude this section with a simple application of the above examples to the nonlinear scattering theory. For the sake of simplicity we only consider the following defocusing nonlinear Schr\"odinger equation with a potential $V$:
\begin{align}
\label{3_1}
i\partial_tu+\Delta u-V(x)u=|u|^{\sigma}u\quad \text{on}\quad \R\times \R^n;\quad u|_{t=0}=u_0\in \H^1,
\end{align}
where we suppose one of the following assumptions (A1)--(A4):
\begin{itemize}
\item[(A1)] $n=1$, $\alpha>0$,\ $V=\alpha\delta$ is the delta potential as in Example \ref{example_3_6} and $\sigma>4$;
\item[(A2)] $n=1$, $\<x\>V,\<x\>V'\in L^1(\R)$, $V\ge0$, $xV'\le0$ and $\sigma>4$;
\item[(A3)] $n=3$, $V\in \mathcal K_0\cap L^{3/2}(\R^3;\R)$, $\norm{V_-}_{\mathcal K}<4\pi$, $\norm{(x\cdot\nabla V)_+}_{\mathcal K}<4\pi$ and $\sigma=2$;
\item[(A4)] $3\le n\le6$, $V=a|x|^{-2}$ and $(n,a,\sigma)$ satisfies
$$\begin{cases}
a>-\frac{(n-2)^2}{4}&\text{if\quad $n=3$ and $4/3<\sigma\le 2$},\\
a>-\frac{(n-2)^2}{4}+\left(\frac{n-2}{2}-\frac{1}{\sigma}\right)^2&\text{if\quad $3\le n\le 6$ and $\max\left(\frac{2}{n-2},\frac 4n\right)<\sigma< \frac{4}{n-2}$}.
\end{cases}
$$
\end{itemize}
Here $\mathcal K$ is the so-called global Kato class with norm 
$$
\norm{V}_{\mathcal K}=\sup_{x\in \R^3}\int_{\R^3}|x-y|^{-1}|V(y)|dy
$$
and $\mathcal K_0$ is the norm closure of bounded compactly supported functions with respect to $\norm{\cdot}_{\mathcal K}$. Note that under one of these conditions (A1)--(A4) the spectrum of $H=-\Delta+V$ is purely absolutely continuous $\sigma(H)=\sigma_{\mathrm{ac}}(H)=[0,\infty)$ and $P_{\mathrm{ac}}(H)=\Id_{L^2(\R^n)}$. 

It was proved by Banica-Visciglia \cite{BaVi} for (A1), Lafontaine \cite{Laf} for (A2), Hong \cite{Hon} for (A3) and Lu-Miao-Murphy \cite{LMM} for (A4), respectively that \eqref{3_1} is globally well-posed in $\H^1$ and the solution $u$ scatters to a linear solution in $\H^1$ in the sense that there exist $\wtilde u_\pm \in \H^1$ such that 
$$
\norm{u(t)-e^{-itH}\wtilde u_\pm}_{\H^1}\to0,\quad t\to \pm\infty.
$$
On the other hand, it follows from Examples \ref{example_3_3} and \ref{example_3_6} that Theorem \ref{theorem_1} with $U_0(t)=e^{-itH_0}$  and $s=1$ holds. Hence, we have the following

\begin{corollary}
\label{corollary_3_7}
Let $u_0\in \H^1$. Then the solution $u$ obtained by the above previous works scatters to a free solution in $\H^1$, namely there exist $u_\pm\in \H^1$ such that
$$
\norm{u(t)-e^{-itH_0}u_\pm}_{\H^1}\to0,\quad t\to \pm\infty.
$$
\end{corollary}

\begin{remark}
(1) With the additional condition $V\in W^{1,3/2}(\R^n)$, Hong \cite{Hon} has proved in case of (A3) that  the solution $u$ scatters to a free solution in $\H^1$. We here do not need such an additional regularity. \\
(2) It was claimed in  \cite{Laf} that $u$ scatters to a free solution in $\H^1$ under the condition (A2). However, the proof in \cite{Laf}  used the same argument as in \cite[Propositions 3.1]{BaVi} in which it was shown that $\{e^{itH}u(t)\}$ is Cauchy in $\H^1$ as $t\to\infty$. This implies the scattering to a linear solution  in $\H^1$, but the scattering to a free solution  in $\H^1$ seems to be not an obvious consequence. 
\end{remark}

\appendix
\section{Some supplementary lemmas}
\label{appendix_A}
\begin{lemma}
\label{lemma_A_1}
Under {\rm (H1)} and {\rm (H2)}, $[\psi(H_0)-\psi(H)]\varphi(H_0)\in \mathbb B_\infty(\H,\H^s)$ for any $\varphi,\psi\in C_0^\infty(\R)$.
\end{lemma}
\begin{proof}
We shall show $\<H_0\>^{s/2}[\psi(H_0)-\psi(H)]\varphi(H_0)\in \mathbb B_\infty(\H)$. Helffer-Sj\"ostrand's formula implies
$$
\<H_0\>^{s/2}[\psi(H_0)-\psi(H)]\varphi(H_0)=\frac{1}{2\pi i}\int_\C\frac{\partial \widetilde{\varphi}}{\partial \overline z}(z)\<H_0\>^{s/2}[(H_0-z)^{-1}-(H-z)^{-1}]\varphi(H_0)dz\wedge d\overline z
$$
where $\widetilde{\varphi}\in C_0^\infty (\R^2)$ is an almost analytic extension of $\varphi$ satisfying $\partial_{\overline z}\widetilde{\varphi}(z)=O(\<z\>^{-N}|\Im z|^N)$ for any $N\ge1$ (\cite{HeSj}). By (H1), (H2), the operator in the integrand is compact and its operator norm is $O(|\Im z|^{-1})$. Hence, the integral converges in norm, being compact on $\H$. 
\end{proof}

\begin{lemma}
\label{lemma_B_1}
Let $\max(n,2)\le r<\infty$, $r>2$ if $n=2$, and $w\in L^r(\R^n)$. Then $w\<D\>^{-1}$ is compact on $L^2(\R^n)$. 
\end{lemma}

\begin{proof}
We recall  Sobolev's inequality 
$
\norm{f}_{L^q(\R^n)}\lesssim \norm{f}_{H^s(\R^n)}$, 
where $s\ge n(1/2-1/q)$ if $2\le q<\infty$ or $s>n/2$ if $q=\infty$. If $r>2$ and $n\ge2$ then
$$
\norm{wf}_{L^2(\R^n)}\le \norm{w}_{L^r(\R^n)}\norm{f}_{L^{\frac{2r}{r-2}}(\R^n)}\lesssim \norm{w}_{L^r(\R^n)}\norm{f}_{H^{n/r}(\R^n)}.
$$
When $r=2$ and $n=1$, we similarly have
$$
\norm{wf}_{L^2(\R)}\le \norm{w}_{L^2(\R)}\norm{f}_{L^{\infty}(\R)}\lesssim \norm{w}_{L^2(\R)}\norm{f}_{H^{1}(\R)}.
$$
Since $n/r\le1$ by assumption, these estimates show
$
\norm{w\<D\>^{-1}}\lesssim \norm{w}_{L^r(\R^n)}.
$
Let $w_n\in C_0^\infty(\R^n)$ be such that $\norm{w-w_n}_{L^r(\R^n)}\to 0$ as $n\to\infty$. Then $w_n\<D\>^{-1}\to w\<D\>^{-1}$ in norm by the above computation. Since $w_n\<D\>^{-1}$ is compact,  $w\<D\>^{-1}$ is also compact.
\end{proof}


\end{document}